\documentclass[reqno, 12pt]{amsart}
\makeatletter
\let\origsection=\section \def\section{\@ifstar{\origsection*}{\mysection}}
\def\mysection{\@startsection{section}{1}\z@{.7\linespacing\@plus\linespacing}{.5\linespacing}{\normalfont\scshape\centering\S}}
\makeatother

\usepackage{amsmath,amssymb,amsthm}
\usepackage{mathrsfs}
\usepackage{mathabx}\changenotsign
\usepackage{dsfont}

\usepackage{xcolor} 

\usepackage[backref]{hyperref}
\hypersetup{
    colorlinks,
    linkcolor={red!60!black},
    citecolor={green!60!black},
    urlcolor={blue!60!black}
}

\usepackage[open,openlevel=2,atend]{bookmark}

\usepackage[abbrev,msc-links,backrefs]{amsrefs}
\usepackage{doi}

\renewcommand{\PrintDOI}[1]{\doi{#1}}

\usepackage[T1]{fontenc}

\usepackage{lmodern}
\usepackage[babel]{microtype}
\usepackage[english]{babel}

\usepackage{geometry}
\numberwithin{equation}{section}
\numberwithin{figure}{section}

\usepackage{enumitem}
\def\rmlabel{\upshape({\itshape \roman*\,})}

\def\alabel{\upshape({\itshape \alph*\,})}

\let\polishlcross=\l
\def\l{\ifmmode\ell\else\polishlcross\fi}

\def\tand{\ \text{and}\ }
\def\qand{\quad\text{and}\quad}
\def\qqand{\qquad\text{and}\qquad}

\let\emptyset=\varnothing
\let\setminus=\smallsetminus

\makeatletter
\def\moverlay{\mathpalette\mov@rlay}
\def\mov@rlay#1#2{\leavevmode\vtop{   \baselineskip\z@skip \lineskiplimit-\maxdimen
   \ialign{\hfil$\m@th#1##$\hfil\cr#2\crcr}}}
\newcommand{\charfusion}[3][\mathord]{
    #1{\ifx#1\mathop\vphantom{#2}\fi
        \mathpalette\mov@rlay{#2\cr#3}
      }
    \ifx#1\mathop\expandafter\displaylimits\fi}
\makeatother

\newcommand{\dcup}{\charfusion[\mathbin]{\cup}{\cdot}}

\DeclareFontFamily{U}  {MnSymbolC}{}
\DeclareSymbolFont{MnSyC}         {U}  {MnSymbolC}{m}{n}
\DeclareFontShape{U}{MnSymbolC}{m}{n}{
    <-6>  MnSymbolC5
   <6-7>  MnSymbolC6
   <7-8>  MnSymbolC7
   <8-9>  MnSymbolC8
   <9-10> MnSymbolC9
  <10-12> MnSymbolC10
  <12->   MnSymbolC12}{}
\DeclareMathSymbol{\powerset}{\mathord}{MnSyC}{180}

\let\epsilon=\varepsilon
\let\eps=\epsilon
\let\theta=\vartheta

\let\phi=\varphi

\def\EE{{\mathds E}}
\def\NN{{\mathds N}}

\def\PP{{\mathds P}}

\newcommand{\cQ}{\mathcal{Q}}
\newcommand{\cR}{\mathcal{R}}

\theoremstyle{plain}
\newtheorem{thm}{Theorem}[section]

\newtheorem{cor}[thm]{Corollary}
\newtheorem{lemma}[thm]{Lemma}

\theoremstyle{definition}

\newtheorem{dfn}[thm]{Definition}

\def\lra{\longrightarrow}
\DeclareMathOperator{\Bin}{Bin}

\begin{document}

\title[Size-Ramsey number of grids]
{On the size-Ramsey number of grid graphs}

\author[D.~Clemens]{Dennis Clemens}
\address{Institut f\"ur Mathematik, TU Hamburg, Hamburg, Germany}
\email{dennis.clemens@tuhh.de} 

\author[M.~Miralaei]{Meysam Miralaei}
\address{Department of Mathematics, Isfahan University of Technology, Isfahan, Iran}
\email{m.miralaei@math.iut.ac.ir}
\thanks{M.~Miralaei was supported by the Ministry of Science, Research and
Technology of Iran and part of the research was carried out during a visit
at University of Hamburg.}

\author[D.~Reding]{Damian Reding}
\address{Institut f\"ur Mathematik, TU Hamburg, Hamburg, Germany}
\email{damian.reding@tuhh.de}

\author[M.~Schacht]{Mathias Schacht}
\address{Fachbereich Mathematik, Universit\"at Hamburg, Hamburg, Germany}
\curraddr{Department of Mathematics, Yale University, New Haven, USA}
\email{schacht@math.uni-hamburg.de}
\thanks{M.~Schacht was partly supported by the 
European Research Council (PEPCo 724903).}

\author[A.~Taraz]{Anusch Taraz}
\address{Institut f\"ur Mathematik, TU Hamburg, Hamburg, Germany}
\email{taraz@tuhh.de}

\begin{abstract}
The size-Ramsey number of a graph $F$ is the smallest number of edges in a graph $G$ with the Ramsey property for $F$, that is, with the property that any 2-colouring of the edges of $G$ contains a monochromatic copy of $F$. We prove that the size-Ramsey number of the grid graph on $n\times n$ vertices is bounded from above by $n^{3+o(1)}$. 
\end{abstract}

\maketitle

\section{Introduction and Results}

For two graphs $F$ and $G$, we 
say that $G$ is \textit{Ramsey for $F$} and write $G\lra F$, 
if every 2-colouring of the edges of $G$ yields  a monochromatic copy of $F$. 
Erd\H{o}s, Faudree, Rousseau, and Schelp~\cite{EFRS} defined the \emph{size-Ramsey number} $\hat{r}(F)$ of $F$ to be the smallest integer $m$ such that there exists a graph $G$ with $m$ edges that is Ramsey for $F$, i.e., 
\[
	\hat{r}(F)=\min\{e(G)\colon G\lra F\}\,.
\] 
Addressing a question posed by Erd\H os~\cite{Er81}, 
Beck~\cite{B83} proved that the size-Ramsey number of the path $P_n$ is linear in $n$ by means of a probabilistic construction and 
Alon and Chung~\cite{AC88} later gave an explicit construction. Beck's proof gave $\hat{r}(P_n)\leq 900 n$ and this upper 
bound was improved several times~\cites{Bo-book,DP,Le16} by simplified and refined probabilistic constructions. 
Currently, the best known upper bound of the form $\hat{r}(P_n)\leq 74 n$ 
is due to Dudek and Pra\l at~\cite{DP-SIAM}. The size-Ramsey number was studied for other graphs than paths including
cycles~\cites{B83,B90,HKL,JKOP}, powers of paths and cycles~\cite{CJKMMRR}, and trees of bounded degree~\cites{B83,B90,FP,HK,Dell}. 
This line of 
research was inspired by a question of Beck~\cite{B90}, whether the size-Ramsey number grows linearly in the number of vertices for 
graphs of bounded degree. In fact, for the graph classes mentioned so far this question was answered affirmatively. However, 
R\"odl and Szemer\'edi~\cite{RSz} gave an example of a sequence of $3$-regular, $n$-vertex graphs $(F_n)_{n\in\NN}$ for which 
they could establish
\[
\hat{r}(F_n)\geq n\log^cn
\]
for some $c>0$.

Moreover, they conjectured that for every $\Delta\ge 3$ there exists an $\eps>0$ such that for every sufficiently large $n$ we have 
\[
	n^{1+\eps} 
	\le 
	\max \hat r(F)
	\le 
	n^{2-\eps}\,,
\]
where the maximum is taken over all $n$-vertex graphs $F$ with maximum degree $\Delta(F)\leq \Delta$.
The upper bound of this conjecture was confirmed by Kohayakawa, R\"odl, Schacht, and Szemer\'edi~\cite{KRSSz}
for any $\eps<1/\Delta$. 
The proof was also based on a probabilistic construction. More precisely, it was shown in~\cite{KRSSz} that for appropriate constants $a>0$ and $C>1$ and~$p=C(\log n/n)^{1/\Delta}$ the random graph $G(n,p)$ asymptotically almost surely has the property $G(n,p)\lra F$
for every $an$-vertex graph $F$ with maximum degree at most $\Delta$.
We remark that the edge probability~$p$ is chosen in such a way that 
any set of $\Delta$ vertices in~$G(n,p)$ has some joint neighbours, which allows 
for a ``greedy type'' embedding strategy for a graph $F$ with maximum degree $\Delta$. 
Recently, Conlon and Nenadov~\cite{CN} managed to overcome this natural barrier for triangle-free graphs $F$ on $n$ vertices 
with $\Delta(F)\leq \Delta$ and $\Delta\geq5$, by showing that
\begin{equation}
\label{eq:CN}
\hat r(F) = O( n^{2-\frac{1}{\Delta-0.5}}\log^\frac{1}{\Delta-0.5} n )\,.
\end{equation}

We focus on $2$-dimensional grids. The $s\times t$ \textit{grid graph} $G_{s,t}$ is defined on the vertex set~$[s]\times [t]$ 
with edges $uv$ present, whenever $u$ and $v$ differ in exactly one coordinate by exactly one. 
For the square grid $G_{n,n}$ on $n^2$ vertices
the upper bounds arising from~\cite{KRSSz} and~\eqref{eq:CN} are of the order $n^{7/2+o(1)}$ and $n^{24/7+o(1)}$, respectively,
if we choose to ignore the restriction $\Delta\ge 5$ in ~(\ref{eq:CN}) for the moment. Our main result improves these upper bounds to $n^{3+o(1)}$ (see Corollary~\ref{cor:main}). 

\begin{thm}
\label{thm:main}
	For all $\alpha' >0$ there exist $c>0$ and $ C\geq 1 $ such that 
for $p\ge C(\log n/n)^{1/2}$
	a.a.s.\ $G\in G(n, p)$ satisfies the following.
	Every subgraph $H\subseteq G$ with $e(H)\geq\alpha' e(G)$ contains a copy of 
	$G_{s, s}$ for any $s\leq \frac{c}{p}$.
\end{thm}
A simple first moment calculation shows that already for $p=cn^{-\frac{1}{2}}$ with $c<1$ the random graph~$G(n,p)$ asymptotically almost surely does not contain a square grid of size logarithmic in~$n$, which shows that the condition on $p$ is almost optimal. 
Theorem~\ref{thm:main} has the following immediate consequence for the size-Ramsey number of square grids. 
\begin{cor}
\label{cor:main} The size-Ramsey number of the $n\times n$ square grid satisfies
$
\hat r( G_{n,n} ) 
\leq 
n^{3 + o(1)}.$
\end{cor}
There is not much evidence that $n^{3}$ is the right order of magnitude for $\hat{r}(G_{n,n})$. 
For the sake of a simpler presentation,
we therefore have made no attempt to strengthen Theorem~\ref{thm:main} in such a way that would allow us to 
remove the $o(1)$ term in the upper bound in 
Corollary~\ref{cor:main}. However, it seems very likely that a more careful analysis in the proof of Theorem~\ref{thm:main} would allow such an improvement.

\section{Preliminaries}
For a graph $G$, we write $V(G)$ and $E(G)$ and $e(G)$ 
for the vertex set, edge set and the number of edges of $G$, 
respectively. 
Given $ v \in V(G) $, by $ N_{G}(v) $ we mean the set of all neighbours 
of $ v $ and set $ d_{G}(v)=|N_{G}(v) | $.
We use the standard notation $\Delta (G)$ and~$\delta (G)$ for the 
maximum and minimum degree of vertices in  $G$, respectively.
For a vertex $ v\in V(G) $ and a subset $ X\subseteq V(G) $, 
let $ N_G(v,X) $ denote the set of neighbours of $ v $ in $X$.
Given a subset~$X\subseteq V(G) $, we let $ G[X] $ be the subgraph of $G$ that is induced by $ X $. 
We write~$G-X $ for $ G[V(G)\setminus X] $.
For subsets $X$, $Y\subseteq V(G)$, we define $G[X,Y]$ to be the subgraph of $G$ on vertex set $X\cup Y$ 
with edges $xy$ where $x\in X$ and $y\in Y$.  We denote by $E_G(X,Y)$ its edge set and set
$e_G(X ,Y):=|E_G(X,Y)|$.

For real numbers $x$, $y$, $\delta >0$, we write $ x=(1\pm \delta)y $ 
if $(1-\delta)y < x < (1+\delta)y,$
and for every integer $k$ we denote by $[k]$ the set of the first $k$ positive integers $\{1,\dots,k\}$.

The binomial random graph $G(n,p)$ is defined on the vertex set $[n]$ that is obtained by pairwise independently
including each of the possible $\binom{n}{2}$ edges with probability $p=p(n)$. We say that an event holds 
\textit{asymptotically almost surely} (abbreviated \textit{a.a.s.}) in $G(n, p)$, if its probability tends to $1$ as $n$ tends to infinity.

The proof of Theorem~\ref{thm:main} is based on the \emph{regularity lemma for subgraphs of sparse random graphs} 
(see Theorem~\ref{reg thm} below), which was introduced by 
Kohayakawa and R\"odl~\cites{MR1661982,MR1980964} and below we introduce the required notation. 
Let $H=(V,E)$ be a graph and let~$p\in(0,1]$ be given. 
Suppose that $K>1$ and $\eta >0 $.
For nonempty subsets $X$, $Y\subseteq V$, we consider the  
\textit{$p$-density} of the pair $(X,Y)$ defined by
 \[
        d_{H,p}(X,Y)=\dfrac{e_H(X,Y)}{p |X| |Y|}\,.
 \]
We say that $H$ is a \textit{$(\eta ,K)$-bounded} with 
respect to the density $p$ if for all pairwise disjoint sets 
$X$, $Y \subseteq V$ with $|X|$, $|Y| \geq \eta |V|$, we have
 \[
		e_H(X,Y)\leq Kp|X| |Y|\,. 
 \]
Given $\eps>0$ and disjoint nonempty subsets $X$, $Y\subseteq V$, 
we say that the pair $(X,Y)$ is \textit{$(\eps ,p)$-regular} if 
for all $X'\subseteq X$ and $Y'\subseteq Y$ satisfying
 \[
 		|X'| \geq \eps |X|  \qand |Y'| \geq \eps |Y|,
 \]
we have
 \[
 		\big| d_{H,p}(X,Y)-d_{H,p}(X',Y')\big| \leq \eps.
 \]
Note that for $p=1$ we recover
the well-known definition of $\eps$-regular pairs in the context of  
Szemer\'{e}di's regularity lemma~\cite{MR540024}.

\begin{dfn}
Given a real number $ \eps >0 $, a positive integer $t$ and a graph $H=(V,E)$, we say that a partition 
$ \{V_{i}\}_{i=0}^{t}$ of $ V $ is
\textit{$ (\eps , p)$-regular} if
 	\begin{enumerate}[label=\rmlabel]
		\item $ |V_{0}|  \leq \eps |V|$,
		\item $|V_{1}| = \dots =| V_{t}|$,
		\item all but at most $\eps \binom{t}{2}$ pairs 
			$(V_{i},V_{j})$ with $ 1\leq i<j \leq t$
	        are $ (\eps ,p)$-regular. 
	\end{enumerate}
The vertex class $ V_{0} $ is referred to as the 
\textit{exceptional set}. 
\end{dfn}
 
The following is a variant of the 
Szemer\'{e}di Regularity Lemma~\cite{MR540024} for sparse graphs.

\begin{thm}[Sparse Regularity Lemma]
\label{reg thm}
	For any $\eps >0$, $K>1$, and $t_{0}\geq 1$, there  exist 
	constants  $T_{0}$, $\eta>0$, and $ N_{0} $ 
	such that any graph $H$ that has least $N_{0}$ vertices and that 
	is $(\eta ,K)$-bounded with respect to some density $p\in(0,1]$, 
	admits an $(\eps ,p)$-regular partition~$\{V_{i}\}_{i=0}^{t}$ 
	of its vertex set with $t_{0}\leq t \leq T_{0}$.\qed
\end{thm}

Considering the random graph $G\in G(n,p)$ it easily follows from Chernoff's inequality
that a.a.s.~$G$ is $(\eta,K)$-bounded with respect to $p$
for any $\eta>0$ and $K>1$ as long as~$p\gg n^{-1}$. 
In such an event, every subgraph $H\subseteq G$ is 
by definition again $(\eta,K)$-bounded with respect to $p$
and consequently it admits a regular partition.
We shall employ the following standard version of
Chernoff's inequality (see, e.g.,~\cite{janson2011random}*{Corollary~2.3}) on the deviation of the binomial random variable~$\Bin(n,p)$.

\begin{thm}[Chernoff's inequality]\label{thm:chernoff}
	For every binomial random variable $X \sim \Bin(n,p)$ 
	and every $\delta\in(0,3/2]$ we have 
	$\PP\big(X \neq (1\pm \delta){\EE}[X]\big) < 2 \exp\big(-\delta^2\EE[X]/3\big)$.
	\qed
\end{thm}

We shall also use the fact that $(\eps,p)$-regularity 
is typically inherited in small subsets, 
which in our setting are given by the neighbourhoods of vertices.
For the classical notion of (dense) $\eps$-regular pairs 
this was essentially observed by Duke and R\"{o}dl~\cite{MR796186} 
and for sparse regular pairs it can be found in~\cites{MR2278123, MR1980964}. More precisely, we
shall employ a result from~\cite{MR2278123} 
governing the hereditary nature of 
\textit{$ (\eps, \alpha, p)$-denseness} (or one sided-regularity).

\begin{dfn}
	Let $\alpha$, $\eps >0 $ and $p\in(0,1]$ be given and 
	let $H=(V,E)$ be a graph. For disjoint, nonempty subsets
	$X$, $Y \subseteq V$, we say that the pair $(X,Y)$ 
	is $(\eps, \alpha, p)$-dense if for all subsets 
	$X'\subseteq X$ and $ Y'\subseteq Y $ with 
	$|X'| \geq \eps |X|$ and $|Y'| \geq \eps |Y|$, 
	we have 
	\[
			d_{H,p}(X',Y')\geq \alpha - \eps.
	\]
\end{dfn}

It follows immediately from its definition that 
$ (\eps, \alpha, p)$-denseness is inherited by large sets, 
i.e. that for an $ (\eps, \alpha, p)$-dense pair $ (X,Y) $ 
and arbitrary subsets $ X'\subseteq X $ and $ Y' \subseteq Y $ 
with $ |X'|\geq \mu |X|$ and $ |Y'| \geq \mu |Y| $ where $\mu\geq \eps$ 
the pair $ (X',Y') $ is $ (\eps/\mu , \alpha, p)$-dense. 
The following result from~\cite{MR2278123}*{Corollary~3.8} states that with exponentially small error probability 
this denseness property is even inherited by
randomly chosen subsets of significantly smaller size.

\begin{thm}\label{thm:Gerke}
Given $\alpha$, $\beta >0$ and $\eps'>0$, 
	there exist constants $\eps_{0}=\eps_{0}(\alpha, \beta, \eps') >0$ 
	and~$L=L(\alpha, \eps') $ such that
	for every $\eps\in (0,\eps_{0}]$ and $p\in(0,1]$, 
	every $ (\eps, \alpha, p)$-dense pair~$(X,Y)$ in a graph $H$ 
	has the following property: the number of pairs $(X',Y')$ of sets 
	with $ X' \subseteq X $ and $ Y' \subseteq Y $ with 
	$|X'|=w_1\geq L/p $ and $|Y'|=w_2\geq L/p$ 
	such that the pair~$(X',Y')$ is not $ (\eps', \alpha, p)$-dense 
	is at most $\beta^{\min \{w_1,w_2\}}\binom{|X|}{w_1} \binom{|Y|}{w_2}$.\qed
\end{thm}

Moreover, we will use the fact that enlarging the sets of some
dense pair by a few vertices may result in a pair that again is dense,
but maybe with slightly weaker parameters (see, e.g.,~\cite{ABHKP}*{Lemma~2.10}).

\begin{lemma}\label{supset}
Let $\alpha>0$, $p\in (0,1)$ and $\eps\in (0,\frac{1}{10})$. 
	Let $H=(V,E)$ be a graph and let $U'$, $W'\subseteq V$ be disjoint nonempty sets 
	such that $(U',W')$ is $(\eps,\alpha,p)$-dense in $H$. 
	If $U\supseteq U'$ and $W\supseteq W'$ are disjoint with 
	$|U|\leq \big(1+\frac{\eps^3}{10}\big)|U'|$ and 
	$|W|\leq \big(1+\frac{\eps^3}{10}\big)|W'|$,
	then $(U,W)$ is $(2\eps,\alpha,p)$-dense in $H$.\qed
\end{lemma}

\section{Properties of random graphs}
\label{sec:gnp}
For the proof of Theorem~\ref{thm:main} we observe a few properties that asymptotically almost surely are satisfied
by the random graph. 
\begin{lemma}\label{lem:gnp}
	For every $\delta>0$ there is some $C>1$ such that for 
	$p=p(n)>C(\log(n)/n)^{1/2}$ a.a.s.\ 
	$G=(V,E)\in G(n,p)$ satisfies the following properties:
	\begin{enumerate}[label=\rmlabel]
	\item\label{it:gnp:1} 
		Every vertex $v\in V$ has degree $d_G(v)=(1\pm \delta)pn$ 
		and the joint neighbourhood of every pair of distinct vertices $u$, $w\in V$ 
        satisfies $|N_G(u)\cap N_G(w)|=(1\pm\delta)p^2n$.
    \item\label{it:gnp:1a}
    	For every vertex $v\in V$ and every subset 
    	$X\subseteq N_G(v)$ with $|X|\geq \delta pn$,
		there are at most $\frac{7}{\delta^3p}$ vertices $y\in V$ 
		such that $|N_G(y)\cap X|>(1+\delta)p|X|$. 
	\item\label{it:gnp:1b}
    	For every subset $U\subseteq V$ with $|U|\geq \delta n$,
		there are at most $\frac{7}{\delta^3p}$ vertices $y\in V$ 
		such that $|N_G(y)\cap U|>(1+\delta)p|U|$. 
	\item\label{it:gnp:2} 
		For every pair of distinct vertices $u,w\in V$ 
		and all subsets $U\subseteq N_G(u)$, $W\subseteq N_G(w)$ with 
		$|U|\geq \delta pn$ and $|W|\geq 3\delta^{-3}pn/\log n$, 
		the number~$e_G(U,W)$ of edges in the induced bipartite graph 
		$G[U,W]$ satisfies $e_G(U,W)= (1\pm\delta)p|U||W|$.
	\item\label{it:gnp:3} 
		For every pair of disjoint sets of vertices $A$, $B\subseteq V$ with
		$|A|$, $|B|\geq \delta n$ we have $e_G(A)=(1\pm \delta)p\binom{|A|}{2}$ and $e_G(A,B)=(1\pm \delta)p|A||B|$.
	\item\label{it:gnp:4}
		For all distinct vertices $u$, $v\in V$ and all disjoint subsets 
		$X\subseteq N_G(u)$, $Y\subseteq N_G(v)$ and disjoint subsets
		$A$, $B\subseteq V$ satisfying $|X|$, $|Y|\geq \delta pn$ and 
		$|A|$, $|B|\geq \delta n$,
		the number of $4$-cycles $xyabx$ with 
		$x\in X$, $y\in Y$, $a\in A$ and $b\in B$
		is bounded from above by $2p^4|X||Y||A||B|$.
	\end{enumerate}
\end{lemma}

\begin{proof}
	Without loss of generality we may assume that $(1+\delta)^4+\delta<2$.
	We set $C=7/\delta^3$ and let $p=p(n)>C(\log(n)/n)^{1/2}$.
	
	By a standard application of Chernoff's inequality  (Theorem~\ref{thm:chernoff}) combined with the union bound 
	it follows that $G\in G(n,p)$ satisfies \ref{it:gnp:1} 
	with probability $1-o(1)$.

	For the proof of part~\ref{it:gnp:1a},
	consider subsets $X$, $Y\subseteq V$ with 
	$|X|=m\geq \delta pn$ and $|Y|=\lceil \frac{7}{\delta^3p}\rceil$. 
	It follows from Chernoff's inequality (Theorem~\ref{thm:chernoff}) 	
	that
	\begin{align*}
		\PP\big(
				|N_G(y)\cap X|>(1+\delta)p|X|\ \text{for every $y\in Y$}\big)
		& \leq
			\PP\left(e_G(X,Y)\neq (1\pm \delta)\EE[e_G(X,Y)]\right)\\
		& < 
			2\exp\big(-\delta^2p|X||Y|/3\big)\\
		& \leq
			2\exp\big(-\tfrac{7\left|X\right|}{3\delta}\big)\,.
	\end{align*}
	Considering all choices of $v$, $Y$, and $X$ and imposing that 
	$X\subseteq N_G(v)$ we arrive at
	\begin{align*}
		\PP(\text{property~\ref{it:gnp:1a} fails})
		& \leq 
			n \cdot n^{\lceil \frac{7}{\delta^3 p} \rceil}
			\cdot \sum_{m\geq \delta p n}
			\binom{n}{m}p^m
			\cdot 2\exp\big(-\tfrac{7m}{3\delta}	
				\big)\\
		& \leq
			2n^{2+\frac{7}{\delta^3 p}}
			\cdot \sum\limits_{m}\left(\frac{epn}{m}\right)^m 
			e^{-\frac{7m}{3\delta}}\\
		& \leq
			2n^{2+\frac{7}{\delta^3 p}}
			\cdot \sum\limits_{m}e^{pn} 
			e^{-\frac{7}{3\delta}\delta pn}
	\end{align*}			
	where we used the fact that $x\mapsto (a/x)^x$ attains its maximum at $x=a/e$.		
	Consequently, in view of the choice of $C$ and $p$ we obtain
	\[
			\PP(\text{property~\ref{it:gnp:1a} fails})
			\leq
			2n^{2+\frac{7}{\delta^3 p}}e^{-\frac{4}{3}pn}
			=o(1)\,.
	\]
	This concludes the proof of part~\ref{it:gnp:1a}.
	
	Part~\ref{it:gnp:1b} follows by a similar argument and we omit the details here.
	
For the proof of part~\ref{it:gnp:2} we
	consider subsets $U$, $W\subseteq V$ satisfying
	$|U|\geq \delta pn$ and $|W|\geq 3\delta^{-3}pn/\log n$
	and vertices $u$, $w\in V$. 
	Applying Chernoff's inequality (Theorem~\ref{thm:chernoff}) 	
	we conclude that
	\[
		\PP\big(e_G(U,W)\!\neq\!(1\pm \delta)p|U||W| \tand U\subseteq N_G(u), W\subseteq N_G(w)\big)
		\leq
			2p^{|U|+|W|}
			\exp\big(-\tfrac{\delta^2}{3}p|U||W| \big) .
	\]
	Considering all choices of $u$, $w$, $U$, and $W$,
	the union bound yields
	\begin{align*}
		\PP(\text{property~\ref{it:gnp:2} fails})
		& \leq 
			n^2 \sum_{m_U\geq \delta p n} \binom{n}{m_U}
				\sum_{m_W\geq \frac{3\delta^{-3} p n}{\log n}}
			\binom{n}{m_W}2p^{m_U+m_W}
			\exp\big(-\tfrac{\delta^2}{3}pm_Um_W \big)\\
		& \leq 
			2n^2 \sum_{m_U} \sum_{m_W}
			\left(\frac{epn}{m_U}\right)^{m_U}
			\left(\frac{epn}{m_W}\right)^{m_W} 
			\exp\big(-\tfrac{\delta^2}{3}pm_Um_W \big)\,.
	\end{align*}
	Appealing again to the fact that $x\mapsto (a/x)^x$ attains its maximum at $x=a/e$
	and the choice of $C$ and $p$ then gives
	\begin{align*}
	\PP(\text{property~\ref{it:gnp:2} fails})
		& \leq
			2n^2 \sum_{m_U} \sum_{m_W}
			e^{2pn}
			e^{-\frac{\delta^2}{3}p\cdot 
				\delta pn \cdot 3\delta^{-3} p n/\log n}  \\
		& =
			2n^4 \left(e^{2-p^2n/\log n}\right)^{pn}\\
		& \leq
			2n^4 \left(e^{2-C^2}\right)^{pn}\\
		&=o(1)\,,
	\end{align*}
	which concludes the proof of part~\ref{it:gnp:2}.

Part \ref{it:gnp:3} again is proven by a standard application of Chernoff's inequality and we omit the proof.

Part~\ref{it:gnp:4} is, in fact, a deterministic consequence of 
properties~\ref{it:gnp:1}--\ref{it:gnp:2}, i.e., 
we will show that every $n$-vertex graph~$G=(V,E)$ 
satisfying~\ref{it:gnp:1}--\ref{it:gnp:2}  
enjoys property~\ref{it:gnp:4}, provided~$n$ is sufficiently large.
	
Let $u$, $v$, $X$, $Y$, $A$ and $B$ be given such that
$A$ and $B$ are disjoint sets with $|A|$, $|B|\geq \delta n$, 
and such that $X\subseteq N_G(u)$, $Y\subseteq N_G(v)$ 
are sets satisfying $|X|$, $|Y|\geq \delta pn$.
We consider the set 
$X'\subseteq X$ of exceptional vertices $x'\in X$ for which  
	\[
		\big|N_G(x')\cap Y\big|>(1+\delta)p |Y|
		\quad\text{or}\quad 
		\big|N_G(x')\cap B\big|>(1+\delta)p |B|\,.
	\]
Similarly, let $Y'\subseteq Y$ be those vertices $y'\in Y$ with too many neighbours in $A$, that is, 
	\[
		\big|N_G(y')\cap A\big|>(1+\delta)p |A|\,.
	\]
It follows from~\ref{it:gnp:2} that the number of $4$-cycles $xyabx$ with $x\in X\setminus X'$, $y\in Y\setminus Y'$, $a\in A$ and $b\in B$ is bounded from above by 
	\[
		|X\setminus X'|\cdot (1+\delta)p|Y|
			\cdot (1+\delta)p\cdot (1+\delta)p|A|\cdot (1+\delta)p|B|
		\leq
		(1+\delta)^4 p^4 |X||Y||A||B|\,.
	\] 
Indeed, fixing an edge $xy\in E_G(X\setminus X',Y\setminus Y')$,
the number of which is bounded from above by 
$|X\setminus X'|\cdot (1+\delta)p|Y|$,
we find at most 
$
(1+\delta)p\cdot |N_G(x)\cap B|\cdot |N_G(y)\cap A|
$
such cycles containing $xy$.

Consequently, it suffices to bound the number of $4$-cycles 
passing through $X'$ or $Y'$ by~$\delta p^4 |X||Y||A||B|$ 
to complete the proof.
To this purpose we note that properties~\ref{it:gnp:1a} 
and~\ref{it:gnp:1b} ensure
	\begin{equation}\label{eq:gnp:C4:1}
		|X'|\leq \frac{14}{\delta^3p}
		\qqand
		|Y'|\leq \frac{7}{\delta^3p}\,.
	\end{equation}
Moreover,~\ref{it:gnp:1} implies that the number of $4$-cycles 
passing through $X'$, $Y$, $A$ and $B$ is at most 
	\[
		|X'|\cdot (1+\delta)p^2n\cdot (1+\delta)pn\cdot (1+\delta)p^2n
	\]
as any vertex $x'\in X'$ has at most $|N_G(x')\cap N_G(v)|\leq (1+\delta)p^2n$ 
neighbours $y\in Y\subset N_G(v)$ and 
at most $(1+\delta)pn$ neighbours~$b$ in~$B$, 
and $y$ and $b$ have at most $(1+\delta)p^2n$ joint neighbours in~$A$.
Similarly, there are at most $(1+\delta)^3p^5n^3|Y'|$ such $4$-cycles passing through $Y'$, and hence it follows from~\eqref{eq:gnp:C4:1} that there are at most 
	\[
		\frac{21}{\delta^3p}(1+\delta)^3p^5n^3
		=
		\frac{21}{\delta^3}(1+\delta)^3p^4n^3
		\leq
		\frac{21}{\delta^3}(1+\delta)^3
			\frac{1}{\delta^4p^2n}p^4|X||Y||A||B|
	\]
$4$-cycles passing through $X'$ or $Y'$, 
where we used $|A|$, $|B|\geq \delta n$ and $|X|$, $|Y|\geq \delta pn$ 
for the last inequality. 
Noting  that $p^2n\to\infty$ as $n\to\infty$ shows that 
there are indeed at most~$\delta p^4 |X||Y||A||B|$ such $4$-cycles 
passing through $X'$ or $Y'$, which concludes the proof of 
part~\ref{it:gnp:4}.
\end{proof}

 The following lemma asserts that a.a.s.\ in $ G(n,p) $, given a sufficiently large bipartite $ (\eps, \alpha, p) $-dense
subgraph $H$ with vertex set $ (A, B) $, say, 
$ (\eps', \alpha, p) $-denseness is inherited by 
most of the pairs $ (N_H(x, B), N_H(y, A)) $ with $ xy \in E_H(A,B) $,
where $\eps$ depends on $\eps'$.

\begin{lemma}\label{lem:matching}
	For every $\gamma$, $\alpha$, $\eps'>0$ there exists $\eps>0 $
	with the property that for every $\eta >0$ 
	there exists $C\geq 1$ 
	such that for $p\geq C(\log(n)/n)^{1/2}$ a.a.s.\
	$G=(V,E)\in G(n, p)$ satisfies the following.
	
	Suppose $ H\subseteq G $ is a bipartite subgraph of $ G $ 
	with vertex set $ V(H)=A\dcup B $ such that 
	\begin{enumerate}[label=\rmlabel]
    	 \item 
     		$\eta n \leq\left|A\right| 
     			\leq \left|B\right|\leq 2|A|$, and 
     		$(A,B)_H$ is $(\eps, \alpha, p)$-dense,
     	\item 
     		for every $ x\in A $ and $ y\in B $ we have 
   	  		\[
     			\left|N_H(y,A)\right|\geq\alpha p\left|A\right|
     			\quad \text{and} \quad
				\left|N_H(x,B)\right|\geq\alpha p\left|B\right|.
      		\]
	\end{enumerate}
	If $ M $ is a matching in $ H $ 
	such that for every edge $xy\in M$ the pair	
		\[
			\left(N_H(y, A),N_H(x,B) \right)_H
		\]
	is not $ (\eps', \alpha, p) $-dense, then $ |M|<\gamma |B|. $
\end{lemma}

\begin{proof}
Let $\gamma, \alpha$ and $ \eps' $ be given. Set
	\[
		\delta = \frac{(\eps')^3}{100}\,, \qquad
		\eps'' = \frac{\eps'}{4}\, \qqand
		\beta =\left(\frac{1}{3}\right)^{\frac{8}{\alpha \delta\gamma}} 
			\left(\frac{\alpha}{2e}\right)^3
	\]
and let $ \eps_0 $ and $ L $ be given by Theorem~\ref{thm:Gerke} 
applied with $ \alpha, \beta $ and $ \eps'' $. We define 
	\[
		\eps = \frac{1}{2}\min\{\eps_0, \alpha\}
	\]
and for given $ \eta >0 $ we set 
	\[
		C=\left(\frac{4}{\eta} \right)^{\frac{1}{2}}.
	\]
Finally, we let $ n $ be sufficiently large.

Suppose that $ H\subseteq G $ is a bipartite subgraph satisfying
conditions $ (i) $ and $ (ii) $ where $V(H)= A \dcup B$. 
Assume that there exists a matching $ M $ 
of size at least $ \gamma |B| $ in $ H $ such that for all edges 
$ xy\in M $ the pair $\left(N_H(y, A),N_H(x,B) \right)_H$
is not $ (\eps', \alpha, p) $-dense.
Then, there exists a matching $M'\subset M$ of size 
	$|M'|\geq \frac{\delta}{4}|M|$
such that $|A\setminus V(M')|\geq (1-\delta)|A| \geq \frac{1}{2} \eta n$, 
$|B\setminus V(M')|\geq (1-\delta)|B|\geq \frac{1}{2} \eta n$, and such that for all 
$xy\in M'$ we have
	\[
		|N_H(y,A\setminus V(M'))| \geq (1-\delta)|N_H(y,A)|
		 \ \tand\
		|N_H(x,B\setminus V(M'))| \geq (1-\delta)|N_H(x,B)|\,.
	\]
Indeed, choosing such a matching randomly by 
including every edge of $M$ independently with probability 
$\frac{\delta}{2}$ into $M'$, 
a simple application of Chernoff's inequality (using Theorem~\ref{thm:chernoff})
shows that the above occurs with probability $1-o(1)$.

Let $A':=A\setminus V(M')$ and $B':=B\setminus V(M')$. Then 
	$(A',B')_H$ is $(2\eps, \alpha, p)$-dense. Moreover,
	applying Lemma~\ref{supset},
	for every $xy\in M'$ the pair 
	$\left( N_H(y,A'), N_H(x,B') \right)_H$
	is not $(\eps'/2,\alpha,p)$-dense,
	since $\left(N_H(y, A),N_H(x,B) \right)_H$
	is not $ (\eps', \alpha, p) $-dense.

Now, fix an edge $ xy \in M' $ such that 
	$ (N_H(y,A'), N_H(x,B'))_H $ is not $ (\eps'/2, \alpha, p) $-dense. 
	It can be verified that there are subsets 
	$ A_y'\subseteq N_H(y,A') $ and 
	$ B_x' \subseteq N_H(x,B') $ of size precisely 
	$ \frac{\eps'}{4} \alpha p |A| $ and 
	$ \frac{\eps'}{4} \alpha p |B|$
	respectively, such that $ d_{H,p}(A_y',B_x')<\alpha - \eps'/2 $. 
Now let $ A_y $ and $ B_x $ be such that 
	$ A_y' \subseteq A_y \subseteq N_H(y,A') $ and 
	$ B_x' \subseteq B_x \subseteq N_H(x,B')$ with 
	$ |A_y|= \frac{1}{2}\alpha p |A| $ and 
	$ |B_x|= \frac{1}{2}\alpha p |B| $.
Then clearly $ (A_y, B_x) $ is not $ (\eps'', \alpha, p) $-dense. 
	We may thus find a family $\{(A_y, B_x)) : xy\in M'\} $ 
	of  pairs of subsets of mentioned size 
	such that these pairs are not $ (\eps'', \alpha, p) $-dense, 
	although $(A',B')_H$ is $(2\eps, \alpha, p)$-dense.

We will now show that a structure consisting of a graph $ H $, 
	disjoint sets $A'$, $B'$ of size at least $\frac{1}{2} \eta n$, a matching $M'$ and a family of pairs $(A_y, B_x)$ 
	as described above is unlikely to appear in $ G(n,p) $.

Since $(A',B')_H$ has to be $(2\eps, \alpha, p)$-dense, we have 
	$ e_H(A', B')\geq (\alpha - 2\eps)p |A'| |B'| $. 
	Thus, we can fix both $ A' $ and $ B' $ along with the edges of the 
	bipartite graph $ H[A',B'] $ in at most
	\[
 		\sum_{|A'||B'|\geq \frac{1}{2}\eta n} 
 			\binom{n}{|A'|} \binom{n}{|B'|}
 		\sum_{t\geq (\alpha -2\eps)p |A'||B'|}  \binom{|A'||B'|}{t}
	\]
ways. 

Note that for $ n $ chosen sufficiently large we have 
	\[
		p^2n\geq C^2\log n \geq \frac{2L}{\alpha \eta}.
	\]
Moreover, owing to the assumption we have 
	$ \frac{1}{2} \alpha p |A|, \frac{1}{2}\alpha p |B|
		\geq \frac{ \alpha p \eta n}{2} \geq L/p. $ 
Hence we can apply Theorem~\ref{thm:Gerke} to $ H[A',B'] $ and 
	infer that there are at most 
\[
	\left[\beta^{\frac{1}{2}\alpha p |A|} 
		\binom{|A'|}{\frac{1}{2}\alpha p |A|} 
		\binom{|B'|}{\frac{1}{2}\alpha p |B|}\right]^{|M'|}
	\leq \left[\beta^{\frac{1}{2}\alpha p |A|} 
		\binom{|A|}{\frac{1}{2}\alpha p |A|} 
		\binom{|B|}{\frac{1}{2}\alpha p |B|}\right]^{|M'|}
\]
	possibilities for choosing the pairs $(A_y, B_x) $ 
	ranging over all $ xy \in M' $,
	when we condition on $(A',B')$ being an 
	$(2\eps,\alpha,p)$-dense pair.
Combining the two above estimates we infer that the probability 
	of the above-described structure appearing in $ G(n,p) $ 
	is bounded from above by the number of choices for $M'$ 
	(easily bounded from above by $e^{n\log n}$) 
	multiplied by
\begin{align*}       
	& \sum_{|A'||B'|\geq \frac{1}{2}\eta n} 
		\binom{n}{|A'|} \binom{n}{|B'|} 
	  \sum_{t\geq (\alpha -2\eps)p |A||B|}  \binom{|A||B|}{t}p^t \\
	& \hspace{2cm} \times  
		\left[\beta^{\frac{1}{2}\alpha p |A|} 
		\binom{|A|}{\frac{1}{2}\alpha p |A|} 
		\binom{|B|}{\frac{1}{2}\alpha p|B|}
		\right]^{|M'|}p^{\frac{1}{2}\alpha p(|A|+|B|)|M'|} \\
    & \leq  
    	\sum_{|A'||B'|\geq \frac{1}{2}\eta n} 
    	\binom{n}{|A'|} \binom{n}{|B'|} 
		\sum_{t\geq (\alpha -2\eps)p |A||B|} 
		\left(\dfrac{|A||B|ep}{t}\right)^{t}\times 
		\left[\beta ^{\frac{1}{2}\alpha p|A|} 
		  \left(\frac{2e}{\alpha}\right)^{\frac{1}{2}\alpha p(|A|+|B|)}
		\right]^{|M'|}\\
 	 & \leq  
 	 	\sum_{|A'|,|B'|\geq \frac{1}{2}\eta n} 
 	 	\binom{n}{|A'|} \binom{n}{|B'|} 
	    \sum_{t\geq (\alpha -2\eps)p |A||B|} e^{p|A||B|} \times
		\left[\beta^{\frac{1}{2}\alpha} 
		 	\left(\frac{2e}{\alpha}\right)^{\frac{3}{2}\alpha}
		 	\right]^{p|A||M'|}\\
	 & \leq  
		\sum_{|A'||B'|\geq \frac{1}{2}\eta n} 
	 	\binom{n}{|A'|} \binom{n}{|B'|} 
		\sum_{t\geq (\alpha -2\eps)p |A||B|} 
			\left(e\beta ^{\frac{1}{2}\alpha \frac{\delta}{4}\gamma} 
				\left(\frac{2e}{\alpha}\right)^{\frac{3}{2}
				\alpha \frac{\delta}{4}\gamma} 	
			\right)^{p|A||B|}\\
	  &\leq 
	  	n^2\cdot 2^{2n}\cdot n^{2} \cdot 
	  		\left( \frac{e}{3} \right)^{n^{3/2}},
 \end{align*} 
where in the second inequality we used the fact that the function 
	$ f(t)= \big(|A||B|ep/t\big)^{t}$ is maximized at 
	$ t=p|A||B| $,
	in the third inequality we used that
	$|M'|\geq \frac{\delta}{4} |M| \geq \frac{\delta \gamma}{4} |B|$,
	and in the fourth inequality we used that 
	$\beta ^{\frac{1}{8}\alpha \delta\gamma} \left(\frac{2e}{\alpha}
		\right)^{\frac{3}{8}\alpha \delta\gamma}=1/3 $ and 
	$p|A||B|\geq n^{3/2}$ for large $n$.
	We conclude that the above-mentioned probability is $o(1)$.
\end{proof}

\section{Technical lemma}
\label{sec:tlemma}
In this section we state and prove the main technical lemma for the proof of Theorem~\ref{thm:main}.
For that we will need the following definition.
\begin{dfn}
	Let $H$ be a graph with disjoint vertex subsets 
	$A$, $B\subseteq V(H)$ and 
	let a set $\eps$, $\alpha$, $\nu$ of constants 
	as well as $p=p(n)>0$ be given.
	\begin{enumerate}[label=\alabel] 
	\item 
		An edge $wz\in E(H)$ is defined to be in 
		$\cR_H(A, B; \eps, \alpha, p)$ if 
		$(N_H(w,A), N_H(z,B))_H$ is $(\eps, \alpha, p)$-dense and 
		$\left|N_H(w,A)\right|\geq\alpha p\left|A\right|$, 
		$\left|N_H(z,B)\right|\geq\alpha p\left|B\right|$.
	\item 
		An edge $wz\in E(H)$ is defined to be in 
		$\cQ_H(A, B; \eps, \alpha, p, \nu)$ if
		\[
			\big|E_H\big(N_H(w,A), N_H(z,B)\big)
				\cap \cR_H(A,B; \eps, \alpha, p)\big|
			\geq 
			(1-\nu)\big|E_H\big(N_H(w,A), N_H(z,B)\big)\big|\,.
		\]
	\end{enumerate}
\end{dfn}

The following lemma will be applied repeatedly in the proof of the main result presented in Section~\ref{sec:pmain}.

\begin{lemma}
	\label{lem:pemb}
	For $\eps'$, $\alpha$, $\mu>0$ with $\eps' < \frac{\alpha}{2}$, 
	there exists $\eps>0$ with the property that 
	for all $\eta>0$ there exists $C>1$ such that for 
	$p\geq C(\log(n)/n)^{1/2}$ 
	a.a.s.\ $G=(V,E)\in G(n, p)$ satisfies the following.

	Suppose $H\subseteq G$ with vertex set 
	$V(H)=X\cup Y\cup A\cup B$ 	
	satisfies
	\begin{enumerate}[label=\rmlabel]
	\item\label{it:tl:1}  
		$X\subseteq N_G(v)$ and $Y\subseteq N_G(u)$ 
		for some vertices $v$, $u\in V$,
	\item\label{it:tl:2} 
		$X\cap Y=\emptyset$ and 
		$\left|X\right|, \left|Y\right|\geq\eta pn$
		and $|E_H(X,Y)| > \frac{\alpha}{2} p|X||Y|$,
	\item\label{it:tl:3} 
		$A\cap B=\emptyset$ and 
		$\eta n \leq |A|\leq |B|\leq 2|A|$, and
	\item\label{it:tl:4} 
		$(A,B)_H$ is $(\eps, \alpha, p)$-dense.
	\end{enumerate}
	If 
	$\left|E_H(X,Y)\cap \cR_{H}(A, B; \eps', \alpha, p)\right|
		\geq (1-\mu) |E_H(X,Y)|$,
	then 
	\[
		\big|E_H(X,Y)\cap \cQ_{H}(A, B; \eps', \alpha, p, \mu) \big|	
		\geq 
		(1-2\mu) \big|E_H(X,Y)\big|\,.
	\]
\end{lemma}

\begin{proof}
Let $\eps'$, $\alpha$, and $\mu>0$ be given with $\eps'<\frac{\alpha}{2}$. 
We fix the auxiliary constant
\begin{equation}\label{eq:tl-gamma}
	\gamma = \frac{1}{50}\alpha^4\mu^2\,.
\end{equation}
For this choice of $\gamma$, $\alpha$, and $\eps'>0$,
Lemma~\ref{lem:matching} yields a constant $\eps>0$, 
which in turn we use for the intended~$\eps$ for Lemma~\ref{lem:pemb}.
Having fixed $\eps>0$, we receive $\eta>0$ and applying 
Lemma~\ref{lem:matching} with the same $\eta$ yields some $C'>1$. Moreover, we apply Lemma~\ref{lem:gnp} with 
\begin{equation}\label{eq:tl-delta}
	\delta=\gamma\eta
\end{equation} 
to obtain some $C''>1$ and we let $C$ be the maximum of 
$C'$ and $C''$.

For $p>C(\log(n)/n)^{1/2}$ and sufficiently large $n$ 
let $H\subseteq G=(V,E)$, where $G\in G(n,p)$, 
and suppose $V(H)=X\cup Y\cup A\cup B$ with 
properties~\ref{it:tl:1}--\ref{it:tl:4} holding. 
To simplify notation we set
\[
	\cR_H=\cR_{H}(A, B; \eps', \alpha, p)
	\qand
	\cQ_H=\cQ_{H}(A, B; \eps', \alpha, p, \mu)\,.
\]
Furthermore, towards a contradiction we assume that 
$|E_H(X,Y)\cap \cR_H|\geq (1-\mu)|E_H(X,Y)|$, but 
\[
	\big|E_H(X,Y)\cap \cQ_{H}\big| < (1-2\mu) \big|E_H(X,Y)\big|\,.
\]
In particular, we have 
\[
 	\big|\big(E_H(X,Y)\cap \cR_{H}\big)\setminus \cQ_{H}\big|
	\geq 
	\mu\big|E_H(X,Y)\big|
	\overset{\text{\ref{it:tl:2}}}
	\geq
	\mu\frac{\alpha}{2}p|X||Y|
	\,. 
\]
In other words, at least $\mu\frac{\alpha}{2}p|X||Y|$ 
edges $xy\in E_H(X,Y)\setminus \cQ_H$ have the property 
that the neighbourhoods $N_H(y,A)$ and $N_H(x,B)$ have size 
at least $\alpha p|A|$ and $\alpha p|B|$, 
respectively, and the pair $(N_H(y,A),N_H(x,B))_H$ is 
$(\eps',\alpha,p)$-dense. Therefore,
\[
	\Big|E_H\big(N_H(y,A),N_H(x,B)\big)\Big|
	\geq 
	(\alpha-\eps')p\big|N_H(y,A)\big|\big|N_H(x,B)\big|
	\geq
	\frac{\alpha^3}{2}p^3|A||B|\,.
\]
However, since $xy\not\in\cQ_H$ at least a $\mu$-fraction 
of those edges are not in $\cR_H$, i.e. 
there exists a subset 
\[
	E'_{xy}\subseteq E_H\big(N_H(y,A),N_H(x,B)\big)\setminus \cR_H
		\quad\text{with}\quad
	\big|E'_{xy}\big|\geq \mu\frac{\alpha^3}{2}p^3|A||B|\,.
\]

Now, Lemma~\ref{lem:matching} along with K\H onig's theorem 
for matchings in bipartite graphs tells us that there are
subsets $A'\subseteq A$ and $B'\subseteq B$ with
$\big|A'\cup B'\big|\leq \gamma |B|$
such that $A'\cup B'$ is a vertex cover for 
$\bigcup E_{xy}'$, where the union is taken over all
$xy\in (E_H(X,Y)\cap \cR_H)\setminus \cQ_H$.
For convenience we fix some supersets 
$A''$ and $B''$ of size $\gamma |B|$ with 
$A'\subseteq A''\subseteq A$ and $B'\subseteq B''\subseteq B$.

Seeing all the edges $xy\in \big(E_H(X,Y)\cap\cR_H\big)\setminus \cQ_H$,
we conclude that there are at least
\begin{align}\label{eq:tl-contra}
	\frac{1}{2} \big|\big(E_H(X,Y)\cap \cR_{H}\big)\setminus \cQ_{H}
		\big| \cdot \big|E'_{xy}\big|
	 & \overset{\phantom{\eqref{eq:tl-gamma}}}{\geq}
	\frac{1}{2} \mu\frac{\alpha}{2}p|X||Y|
		\cdot\mu\frac{\alpha^3}{2}p^3|A||B| \nonumber \\
	& \overset{\eqref{eq:tl-gamma}}{>}
	6\gamma p^4 |X||Y| |A||B|
\end{align}
$4$-cycles $xyabx$ in $H\subseteq G$ with 
$x\in X$, $y\in Y$, $a\in A$, and 
$b\in B$, where $a\in A'$ or $b\in B'$.

On the other hand, Lemma~\ref{lem:gnp}~\ref{it:gnp:4} 
applied to  $X$, $Y$, $A''$, and $B$ (see also~\eqref{eq:tl-delta}),
asserts that  there are at most $2 p^4 |X||Y| |A''||B|$ such cycles
passing through $A'\subseteq A''$ and, similarly, 
there are at most $2p^4 |X||Y| |A||B''|$ such cycles 
passing through $B'\subseteq B''$. 
Owing to $|A''|=|B''|=\gamma|B|\leq 2\gamma |A|$,
it follows that there are at most $6\gamma p^4 |X||Y| |A||B|$
such $4$-cycles, which contradicts~\eqref{eq:tl-contra} 
and concludes the proof.
\end{proof}

By a similar argument we may ignore the conditions $ (i) $ and $ (ii) $ 
in Lemma \ref{lem:pemb} and consider the sets $ X $ and $ Y $ to be of order $ \Omega (n) $.
This way we obtain the following version of Lemma~\ref{lem:pemb}.
\begin{lemma}\label{cor:pemb}
	For all $\eps', \alpha, \mu>0$, there exists $\eps>0$ such that 
	for all $\eta>0$ there exists $C>1$ such that for 
	$p\geq C(\log n/n)^{1/2}$ 
	a.a.s.\ $G=(V,E)\in G(n, p)$ satisfies the following.

	Suppose $H\subseteq G$ with vertex set $V(H)=X\cup Y\cup A\cup B$ 
	satisfies
	\begin{enumerate}[label=\rmlabel]
	\item 
		$X\cap Y=\emptyset$ and 
		$\left|X\right|, \left|Y\right|\geq\eta n$ and 
		$|E_H(X,Y)| > \frac{\alpha}{2} p|X||Y|$,
	\item 
		$A\cap B=\emptyset$ and 
		$\left|A\right|, \left|B\right|\geq\eta n$, and
	\item 
		$(A,B)_H$ is $(\eps, \alpha, p)$-dense.
	\end{enumerate}
	If 
	$\left|E_H(X,Y)\cap \cR_{H}(A, B; \eps', \alpha, p)\right|
		\geq (1-\mu) |E_H(X,Y)| $,
	then 
	\[
		\left|E_H(X,Y)\cap \cQ_{H}(A, B; \eps', \alpha, p, \mu) \right|	
		\geq (1-2\mu) |E_H(X,Y)|.
	\]
\end{lemma}

\section{Proof of the main Result}\label{sec:pmain}
In this section, we prove our main result, Theorem~\ref{thm:main}.

\begin{proof}
The proof consists of three parts. 
	In the first part we fix all constants needed for the proof.
	In the second part we assume that $ G\in G(n,p) $ and 
	$ H\subseteq G $ satisfy $ |V(H)|=n $ and $e(H)\geq\alpha' e(G)$,
	so as to find a suitable dense pair $(A,B)$ in $H$ 
	among which we aim to embed the grid $ G_{s, s} $. 
	Here the Sparse Regularity Lemma will play a key role.
	In the last part we find the embedding.

\medskip

\noindent \textbf{Constants}. 
Let $ \alpha'>0 $ be given. First we define 
	the constants $\mu$, $\alpha$, and $\eps'$ of 
	Lemmas~\ref{lem:pemb} and \ref{cor:pemb}. We set 
	\begin{eqnarray}\label{main.c1}
		\alpha = \alpha'/20, \quad 
\mu = \alpha / 100
	\end{eqnarray}
	as well as
	\begin{eqnarray}\label{main.c2}
		\eps'=\min \{\alpha/4, 1/2 \}
	\end{eqnarray}
	and let $ \eps_{1}>0 $ and $ \eps_{2}>0 $ 
	be as guaranteed by Lemma~\ref{lem:pemb} and Lemma~\ref{cor:pemb}, 	
	respectively. Further, we fix 
	\begin{eqnarray}\label{main.c3}
		\gamma =\frac{\alpha \mu}{8}.
	\end{eqnarray}
	Applying Lemma~\ref{lem:matching} with 
	$ \gamma, \alpha, \eps' $ as defined yields a constant
	$ \eps_{3}>0 $.
	Next we set 
	\begin{eqnarray}\label{main.c4}
		\eps=
		 \min\left\{\eps_{1},\eps_{2}, \eps_{3}, 
		 	\frac{\alpha}{3}, 1/4\right\},~~
	K=2~~ \qand t_{0}=\frac{2}{\alpha}.
	\end{eqnarray}
	Let $ T_{0},\lambda$ and $ N_{0} $ be the constants guaranteed 
	by the Sparse Regularity Lemma (Theorem~\ref{reg thm}) 
	corresponding to $\frac{\eps}{2}, K,$ and $t_{0} $ given above. 
	We set
	\begin{eqnarray}\label{main.c5}
		\eta=\frac{1}{2T_{0}}.
	\end{eqnarray}
Having $ \eps_{1},\eps_{2}, \eps_{3}>0 $
	and applying Lemmas~\ref{lem:pemb},
	\ref{cor:pemb} and \ref{lem:matching} respectively 
	with this $\eta $ yields some constants
	$ C_{1}, C_{2}$ and $ C_{3} $.
	Moreover, let
	\begin{eqnarray}\label{main.c6}
		\delta =\min\{1/10 , \mu\alpha^3\eta^2 \},
	\end{eqnarray}
	and apply Lemma~\ref{lem:gnp} with
	\[
		\delta'=\min\{\eps \alpha \delta \eta,\lambda\} ,
	\] 
	obtaining some constant $ C_{4}>1 $.
	We set $ C:=\max \{C_{1}, C_{2}, C_{3}, C_{4} \} $. 
	Finally, we set
	\begin{equation}\label{main.c7}
		c=\delta\alpha\eta /16
	\end{equation}
	and consider $n$ to be large enough whenever necessary.

\noindent \textbf{Finding a good dense pair.} 
Let $ G=(V,E)\in G(n,p) $, where  
	$ p>C(\log n/n)^{1/2} $. 
	We condition on $G$ satisfying the properties from Lemma~\ref{lem:gnp}.
	Assume that $ H$ is an $ n $ vertex subgraph of $ G$ with 
	$ e(H)\geq \alpha'e(G) $.

	We apply the Sparse Regularity Lemma (Theorem~\ref{reg thm}) 
	with $\eps/2$, $K=2$, $t_0$, and $ p $ to $ H $. 
Note that, owing to Lemma~\ref{lem:gnp}~\ref{it:gnp:3}, the graph $ G $
	and hence $ H\subseteq G $ are $ (\lambda, K)$-bounded. 
Theorem~\ref{reg thm} therefore yields a constant $T_0$ and an $ (\eps/2 ,p)$-regular partition
	$ \{V_{i}\}_{i=0}^{t} $ of $V(H)$ with 
	$ t_{0}\leq t \leq T_{0}. $

	By Lemma~\ref{lem:gnp}~\ref{it:gnp:1} the exceptional set touches 
	at most $2pn|V_0|\leq \eps pn^2$ edges. 	
	Also by Lemma~\ref{lem:gnp}~\ref{it:gnp:3}
	there are at most $2p|V_i||V_j|\leq 2p (n/t)^2$ edges
	between non-regular pairs $(V_i,V_j)$ with  $1\leq i<j\leq t$ and
	inside each of the partition sets $V_i$ there are at most 
	$2p\binom{|V_i|}{2}\leq p (n/t)^2$ edges.
	Thus the number of edges in $H$ 
	both inside the partition sets and between non-regular pairs 		
	is bounded from above by
	\[
		\eps pn^2 + \frac{\eps}{2} \binom{t}{2} \cdot 2p 
			\left(\frac{n}{t}\right)^2 
		+ t \cdot p\left(\frac{n}{t}\right)^2
		< 
		\left(\frac{3}{2}\eps + \frac{1}{t_0} \right) pn^2 
		\stackrel{(\ref{main.c4})}{<} \frac{\alpha'}{8} pn^2 \ .
	\]
	By (\ref{main.c6}) and Lemma~\ref{lem:gnp}~\ref{it:gnp:3} 
	we have $e(H)\geq \alpha' e(G) \geq \alpha'pn^2/4$. 
	Hence, the number of edges lying in $ (\eps/2 ,p)$-regular 
	pairs is at least $\alpha'pn^2/8$.\\
Averaging now guarantees the existence of an 
	$ (\eps/2 ,p)$-regular pair $(V_i,V_j)$ such that
	\[
		e_H(V_i,V_j) 
		\geq \frac{\frac{\alpha'}{8} pn^2}{\binom{t}{2}} 
		> \frac{\alpha'}{4} p |V_i||V_j|\ .
	\]
	Thus, $(A_1,B_1):=(V_i,V_j)$ is $(\eps/2, \alpha'/4, p)$-dense.

	Next, we will discard vertices of too small or too large 
	degree inside the pair $(A_1,B_1)$. First, set  
	\begin{align*}
		A_2 & :=\left\{v\in A_1:\ 
			|N_H(v,B_1)|< \frac{\alpha'}{8} p|B_1|\right\},\\
	 	B_2 & :=\left\{v\in B_1:\ 
			|N_H(v,A_1)|< \frac{\alpha'}{8} p|A_1|\right\}.
	\end{align*}
	Then $|A_2|\leq \frac{\eps}{2} |A_1|$ and 
	$|B_2|\leq \frac{\eps}{2} |B_1|$, 
	as $(A_1,B_1)$ is $(\eps/2, \alpha'/4, p)$-dense.
	Next set
	\begin{align*}
		A_2' & :=\left\{v\in A_1:\ 
			|N_H(v,B_1)|> (1+\delta) p|B_1|\right\},\\
		B_2' & :=\left\{v\in B_1:\ 
			|N_H(v,A_1)|> (1+\delta) p|A_1|\right\}.
	\end{align*}
	We then observe that 
	$\left|A_2'\right|
		\leq \frac{7}{(\delta')^3 p} 
		< \frac{\eps}{4}|A_1|				
	$	
for large $n$, by applying Lemma~\ref{lem:gnp}~\ref{it:gnp:1b}
	and using that 
	$|A_1|\geq \frac{(1-\eps)n}{T_0}>\eta n> \delta' n$
	by the choice of $\eta$ and $\delta'$.
	Similarly, $\left|B_2'\right| < \frac{\eps}{4}|B_1|$ holds.
	Finally, set
	\begin{align*}
	A_3 & :=\left\{v\in A_1\setminus (A_2\cup A_2'):\ 
		|N_H(v,B_2\cup B_2')|\geq \frac{\alpha'}{16} p|B_1|\right\},\\
	B_3 & :=\left\{v\in B_1\setminus (B_2\cup B_2'):\ 
		|N_H(v,A_2\cup A_2')|\geq \frac{\alpha'}{16} p|A_1|\right\}.
	\end{align*}
	Then, considering an arbitrary set 
	$\tilde{B}\supset B_2\cup B_2'$
	with $|\tilde{B}|=\eps |B_1|\geq \eps \eta n > \delta' n$,
	and observing that by definition every $v\in A_3$ satisfies 
	\[
		|N_H(v,\tilde{B})|\geq \frac{\alpha'}{16}p|B_1|>3p|\tilde{B}|
	\]
	by the choice of $\eps$ in (\ref{main.c4}),
	Lemma~\ref{lem:gnp}~\ref{it:gnp:1b} ensures that
	$|A_3|\leq \frac{7}{(\delta')^3 p}<\frac{\eps}{4}|A_1|$. 
	Analogously, $|B_3|<\frac{\eps}{4}|B_1|$ holds. 			
	Now set 
	$$A:=A_1\setminus (A_2\cup A_2'\cup A_3)
	~~ \text{and} ~~	
	B:=B_1\setminus (B_2\cup B_2'\cup B_3)~ .$$ 
	By the choice of $ \eta $ in (\ref{main.c5}) and 
for large enough $n$
	we have
	$|A|\geq (1-\eps)|A_1| \geq \eta n$ and 
	$|B|\geq (1-\eps)\left|B_1\right|\geq\eta n$. Also, since 
	$(A_1,B_1)$ is $(\eps/2, \alpha'/4, p)$-dense,
	the pair $(A,B)_H$ is $(\eps, \alpha, p)$-dense and, 
	by the definition of $A$ and $B$, we obtain that for every vertex
	$v\in A$ we have 
	\begin{equation}\label{main.c8}
		|N_H(v,B)|
		\geq \frac{\alpha'}{8} p|B_1| 
			- \frac{\alpha'}{16} p|B_1| - |B_3| 
		> \alpha p |B|
	\end{equation}
	and for every vertex $v\in B$ we have 
	$|N_H(v,A)|\geq \alpha p |A|$.
	Without loss of generality let $|A|\leq |B|$ and note that $|B|\leq 2|A|$ holds.

\noindent \textbf{Embedding} $G_{s,s}$.
	Recall that $s=\frac{c}{p}$.
	From now on, we fix the pair $(A,B)_H$
and aim to embed the grid $G_{s,s}$
	iteratively in the bipartite graph $ H[A,B] $. Let
	\[	
		\cR_H :=\cR_{H}(A, B; \eps', \alpha, p) 
		\quad\text{and}\quad
		\cQ_H:=\cQ_{H}(A, B; \eps', \alpha, p, \mu) .
	\]
Towards this purpose, we say that a sequence of paths 
	$(P_1,\ldots,P_t)$ in $H$ \textit{produces a copy} of $G_{s,t}$ in $H$ 
	if $|V(P_i)|=s$ holds for every $i\in [t]$,
	and if between each of the pairs $(P_i,P_{i+1})$ with $i\in [t-1]$
	there exists a matching $M_i$ in $H$ such that
	$\bigcup_{i\in [t]} E(P_i) \cup \bigcup_{i\in [t-1]} M_i$
	induces a copy of $G_{s,t}$.
	
	We now prove the following inductively for every $t\in [s]$:
	there exists a sequence $(P_1,\ldots,P_t)$ of paths in $H[A,B]$ 
	such that the following is true:
	
	\begin{enumerate}[label={(P\arabic*)}]
	\item\label{it:P1} $(P_1,\ldots,P_t)$ produces a copy 
		of $G_{s,t}$ on some vertex set $S\subset A\cup B$,
	\item\label{it:P2} $E(P_t)\subset \cR_H\cap \cQ_H$,
	\item\label{it:P3} for every $v\in V(P_t)\cap A$ we have 
		$|N_H(v,B\setminus S)|\geq (1-\delta)|N_H(v,B)|$,
	\item\label{it:P4} for every $v\in V(P_t)\cap B$ we have 
		$|N_H(v,A\setminus S)|\geq (1-\delta)|N_H(v,A)|$.
	\end{enumerate}

\smallskip

\noindent
{\it Induction start:} 
Since $(A,B)_H$ forms an $(\eps,\alpha,p)$-regular pair, we have
	\[
		|E_H(A,B)|\geq (\alpha - \eps)p|A||B|\,.
	\] 
	In particular, 
	we claim that
	\[ 
		|E_H(A,B)\cap \cR_H| \geq (1-\mu)|E_H(A,B)|. 
	\]
	Indeed, if this were not true, we would have 
	$|E_H(A,B)\setminus \cR_H| > \frac{\alpha\mu}{2} p|A||B|$.
	However, recalling the definition of $A$ and $B$, we can bound the maximum degree
	\[
		\Delta(H[A,B]) 
		\leq 
		(1+\delta)p\cdot\max\{|A_1|,|B_1|\} 
		\leq 
		\Big(\frac{1+\delta}{1-\eps}\Big)p\cdot\max\{|A|,|B|\} 
		< 4p|A|
	\]
and we would then find a matching of size at least
	$\frac{\alpha \mu p|A||B|}{8p|A|} \geq \gamma |B|$
	in $E(H)\setminus \cR_H$, 
	which contradicts Lemma~\ref{lem:matching}.

	Applying Lemma~\ref{cor:pemb} (with $X=A$ and $Y=B$)
	we now obtain that
	\[ 
		|E_H(A,B)\cap (\cR_H\cap \cQ_H)| \geq (1-3\mu)|E_H(A,B)| .
	\]
	Therefore, on average the vertices in $A\cup B$ 
	are incident with at least
	\begin{equation*}
		\frac{2(1-3\mu)|E_H(A,B)|}{|A\cup B|} 
		\geq \frac{\frac{\alpha}{2}p|A|^{2}}{3|A|}
		\geq \frac{\alpha \eta p n}{6} 
> \frac{2c}{p}
	\end{equation*}
	edges from $E_H(A,B)\cap \cR_H\cap \cQ_H$.
Thus, we can find a path $P_1$ with 
	\begin{equation}\label{eq:defs}
		s=\frac{c}{p}
	\end{equation}
	vertices, 
	consisting of edges in $\cR_H\cap \cQ_H$ only,
	which gives the properties~\ref{it:P1} and~\ref{it:P2} for $t=1$.
	By (\ref{main.c7}) and (\ref{main.c8}) 
	for every $v\in V(P_1)\cap A$ we have 
	\[
		|N_H(v,B\setminus V(P_{1}))|
		\geq |N_H(v, B)| - c/p \geq (1-\delta)|N_H(v,B)|.
	\]
	Hence property~\ref{it:P3} and then similarly property~\ref{it:P4} follow.
	
	\smallskip
	
	\noindent
	{\it Induction step:} 
	Assume we have found a sequence $(P_1,\ldots,P_t)$ 
	satisfying (P1)-(P4) with $t<s$.
	We aim to extend the sequence by another path $P_{t+1}$.
	Let $\{x_1,x_2,\ldots,x_s\}$ denote the vertices of $P_t$ 
	with $x_ix_{i+1}\in E(P_t)$ for every $i\in [s-1]$.
	In order to find a path $P_{t+1}$, 
	we first fix suitable candidate sets $X_i$
	for embedding the unique neighbour of each $x_i$ 
	on the path $P_{t+1}$.
	We want these candidate sets to be pairwise disjoint, 
	which is why in the following we
	exclude from the neighbourhoods $N_H(x_i)$ the intersections 
	$N_i$ with all the other relevant neighbourhoods.
	Moreover, we want every vertex in $X_i$ to have a suitable
	degree outside the set $S=\bigcup_{i\in [s]} V(P_i)$
	to be able to extend the embedding of the grid later on;
	thus we also exclude a set $S_i$ which contains 
	those vertices that have too large degrees towards
	$S$.
		
	Without loss of generality, we may suppose that $x_{i}\in B$ and $x_{i+1}\in A$.
	Note that $(N_H(x_i,A),N_H(x_{i+1},B))_H$ 
	is $(\eps',\alpha,p)$-dense for every $i\in [s-1]$ since
	$x_ix_{i+1}\in \cR_H$. 
	For every $i\in [s]$, we now consider
	\begin{align*}
		N_i & := N_H(x_i,A\setminus S)\cap 
			\bigg( \bigcup_{j\neq i} N_H(x_j,A) \bigg)\,,\\
		S_i & := \left\{v\in  N_H(x_i,A):\ 
			|N_H(v,S)|>\frac{\delta}{2} \alpha p |A|\right\}\,,\\
		X_i & := N_H(x_i,A\setminus S)\setminus (N_i\cup S_i)\,.
	\end{align*}
	Note that this notation depends on whether or not $ x_i\in A $ 
	or $ x_i\in B $. We have defined them on the assumption that 
	$ x_i \in B $. 
	If $ x_i \in A $, then one should replace $ A $ by $ B $.
	
	Applying Lemma~\ref{lem:gnp}~\ref{it:gnp:1}, property~\ref{it:P3}, and~\eqref{main.c8}, we have
	\[
		|N_i|\leq s\cdot 2p^2n 
		\overset{\eqref{eq:defs}}{=}
		2cpn 
		\overset{\eqref{main.c7}}{=} 
		\frac{\delta\alpha\eta}{8}pn
		<
		\frac{\delta}{2}\left|N_H(x_i,A\setminus S)\right|.
	\] 
	Considering an arbitrary set 
	$\tilde{S}\supset S$
	with $|\tilde{S}|=\frac{\delta\alpha\eta}{4} n\geq \delta' n$,
	and observing that by definition every $v\in S_i$ satisfies 
	\[
		|N_H(v,\tilde{S})|\geq |N_H(v,S)|
			\geq \frac{\delta}{2}\alpha p |A|
			> \frac{\delta}{2}\alpha \eta pn
			=2p|\tilde{S}|,
	\]
	Lemma~\ref{lem:gnp}~\ref{it:gnp:1b} together with property~\ref{it:P3} and
	(\ref{main.c8}) ensures
	\[
		|S_i|\leq \frac{7}{(\delta')^3 p}< \frac{\delta}{2}|N_H(x_i,A\setminus S)|
	\]
	for sufficiently large~$n$. Thus,
	\[
		\left|X_i\right|
		\geq \left(1-\delta \right)\left|N_H(x_i,A\setminus S)\right|.
	\]
	Moreover, all the sets $X_i$ are pairwise disjoint,
	the pairs $(X_i, X_{i+1})$ are $(2\epsilon', \alpha, p)$-dense 
	for every $i\in [s-1]$
	and for every $v\in X_i$ we know that 
	\[
		|N_H(v, B\setminus S)|>
			\left(1-\delta \right)|N_H(v, B)|.
	\]
	Therefore, properties~\ref{it:P3} and~\ref{it:P4} will hold,
	once we manage to find a path $P_{t+1}$ 
	with one vertex from each $X_i$,
	consisting of edges from $\cR_H\cap \cQ_H$ only. 
	Let $i\in [s-1]$. As $x_ix_{i+1}\in \cQ_H$, we have
	\begin{eqnarray*}
		\left|E_H(N_H(x_i,A), N_H(x_{i+1},B))\cap \cR_H\right|
			\geq (1-\mu)|E_H(N_H(x_i,A), N_H(x_{i+1},B))|
	\end{eqnarray*}
	and by applying Lemma~\ref{lem:pemb}	
	we obtain
	\begin{eqnarray*}
		\left|E_H(N_H(x_i,A), N_H(x_{i+1},B))\cap 
			(\cR_H\cap \cQ_H)\right|
		\geq (1-3\mu)|E_H(N_H(x_i,A), N_H(x_{i+1},B))|.
	\end{eqnarray*}
	Moreover, using that $x_ix_{i+1}\in \cR_H$ and (\ref{main.c8}), 
	we get
	\begin{align}\label{main.c9}
		|E_H(N_H(x_i,A), N_H(x_{i+1},B))| 
			& \geq (\alpha - \eps')p |N_H(x_i,A)||N_H(x_{i+1},B)| 		
					\nonumber \\
			& \geq \frac{\alpha^2 \eta}{2} p^2 n
					\cdot \max\{|N_H(x_i,A)|,|N_H(x_{i+1},B)|\}.
	\end{align}	
	By Lemma~\ref{lem:gnp}~\ref{it:gnp:1} 
	every vertex in $N(x_i,A)$ has at most $2p^2n$ neighbours in 			$N(x_{i+1},B)$, and vice versa.
	Combining this with (\ref{main.c9}) we then know that 
	\begin{align*}
		|E_H(N_H(x_i,A), N_H(x_{i+1},B))| - |E_H(X_i,X_{i+1})| 
			& \leq  
			2 p^2n (\delta |N_H(x_i,A)|+ \delta |N_H(x_{i+1},B)|)\\ 
			& \leq  
			4\delta p^2n \max\{|N_H(x_i,A)|,|N_H(x_{i+1},B)|\}\\
			& <  
			\mu |E_H(N_H(x_i,A), N_H(x_{i+1},B))| 
	\end{align*}
	owing to the choice of $\delta$ in (\ref{main.c6}).
	Thus, we conclude that many edges in $E_H(X_i,X_{i+1})$ 
	belong to $\cR_H\cap  \cQ_H$ in the sense that
	\begin{eqnarray}\label{main.c10}
		\left|E_H(X_i,X_{i+1})\cap (\cR_H\cap \cQ_H)\right|
			\geq (1-4\mu)|E_H(N_H(x_i,A), N_H(x_{i+1},B))|\ .
	\end{eqnarray}
	In order to find the desired path we now prove a 
	slightly stronger statement:
	in every set~$X_i$ at least half of its vertices can be reached 
	from~$X_1$ via a path in $\cR_H\cap \cQ_H$. For this purpose, 			
	we iteratively define
	\[
		X_i':=
		\begin{cases}
			X_1 & \text{if }i=1,\\
			\left\{v\in X_i\colon\, \exists w\in X_{i-1}'\ 
				\text{s.t. }vw\in \cR_H\cap \cQ_H\right\} 
				& \text{if }i>1\ .	 
		\end{cases}
	\]
	By induction, we then show that $|X_i'|\geq \frac{1}{2}|X_i|$ 
	for every $i\in [s]$.
	Note that once this is proven, we are done, as we can then take a 
	path $P_{t+1}$ consisting of one vertex from every~$X_i'$ 
	and edges from $\cR_H\cap \cQ_H$ only,
	such that all of the properties~\ref{it:P1}-\ref{it:P4} are satisfied.

	The case $i=1$ is trivial. So, let $i>1$ and assume for 
	a contradiction that $\tilde{X}_i:=X_i\setminus X_i'$ has size at 
	least $\frac{1}{2}|X_i|$.
	By definition of~$X_i'$ we have 
	$E_H(X_{i-1}',\tilde{X}_i)\cap (\cR_H\cap \cQ_H) = \varnothing$ 
	and thus
	\begin{align*}
		|E_H(X_{i-1}', \tilde{X}_i)| 
			\leq |E_H(X_{i-1},X_i)\setminus (\cR_H\cap \cQ_H)|
		& \leq 4\mu \left|E_H(N(x_{i-1},B), N(x_i,A))\right| \\
		& \leq 5\mu p\left|N(x_{i-1},B)\right|\left|N(x_i,A)\right|,
	\end{align*}
	where in the second inequality we use (\ref{main.c10}),
	and where in the last inequality we apply 
	Lemma~\ref{lem:gnp}~\ref{it:gnp:2}.
	However, as	$(X_{i-1}, X_i)$ is $(2\epsilon', \alpha, p)$-dense
	and $|X_{i-1}'|\geq \frac{1}{2}|X_{i-1}|$ (by induction) 
	and also $|\tilde{X}_i|\geq \frac{1}{2}|X_i|$ (by assumption),
	we must have
	\begin{align*}
	|E_H(\tilde{X}_i, X_{i-1}')| 
		 \geq (\alpha-2\epsilon')p
		 	\left|\tilde{X}_i\right|\left|X_{i-1}'\right|
		& \geq \frac{\alpha-2\epsilon'}{4}p
			\left|X_i\right|\left|X_{i-1}\right|\\
		& \geq \frac{\alpha-2\epsilon'}{4}(1-\delta)^2 p
			\left|N(x_{i-1},B\setminus S)\right| \left|N(x_i,A\setminus S)\right|\\
		& \geq \frac{\alpha-2\epsilon'}{4}(1-\delta)^4 p
			\left|N(x_{i-1},B)\right| \left|N(x_i,A)\right|\\
		& > 5\mu p\left|N(x_{i-1},B)\right|\left|N(x_i,A)\right|,
	\end{align*} 
	a contradiction. Hence, 
	$\left|X_i'\right|\geq\frac{1}{2}\left|X_i\right|$	 
	for every $i\in [s]$.
	\end{proof}

\subsection*{Acknowledgement} We are grateful to Thomas Lesgourgues for pointing out an inaccuracy in an earlier version of the manuscript.

\end{document}